\documentclass[11pt,reqno]{amsart}

\usepackage[T1]{fontenc}
\usepackage[utf8]{inputenc}
\usepackage{amsmath,amssymb,amsthm,mathtools}
\usepackage{booktabs}
\usepackage{longtable}
\usepackage{tikz}
\usepackage{array}
\usepackage{microtype}
\usepackage[margin=1.15in]{geometry}
\usepackage[colorlinks=true,linkcolor=blue!55!black,citecolor=blue!55!black,
            urlcolor=blue!55!black]{hyperref}

\usetikzlibrary{calc}

\theoremstyle{plain}
\newtheorem{theorem}{Theorem}[section]
\newtheorem{proposition}[theorem]{Proposition}
\newtheorem{lemma}[theorem]{Lemma}
\newtheorem{corollary}[theorem]{Corollary}
\theoremstyle{definition}

\newtheorem{remark}[theorem]{Remark}

\DeclareMathOperator{\conv}{conv}
\DeclareMathOperator{\rank}{rank}
\DeclareMathOperator{\sig}{sign}
\newcommand{\HH}{\mathbb{H}}
\newcommand{\RR}{\mathbb{R}}
\newcommand{\QQ}{\mathbb{Q}}
\newcommand{\ZZ}{\mathbb{Z}}
\newcommand{\Gr}{G}
\newcommand{\PP}{P_{6,10}}

\newcommand{\NumOrderTypes}{14{,}309{,}547}
\newcommand{\NumTypes}{387}

\newcommand{\NumDegKilled}{76}

\newcommand{\DFiveWildSurvived}{9}

\newcommand{\DFiveWildCert}{2{,}400}

\newcommand{\DFiveWildScreenPass}{11}

\newcommand{\NumPlain}{82{,}840}
\newcommand{\NumWild}{75{,}363}

\newcommand{\NumAssignments}{158{,}203}
\newcommand{\CPUHours}{13}
\newcommand{\RealizingType}{379}
\newcommand{\RealizingP}{3}
\newcommand{\RealizingProfile}{2^{3}3^{3}5^{3}}

\newcommand{\NumGramConfigs}{12}
\newcommand{\AutTypeOrder}{24}
\newcommand{\AutStabOrder}{2}

\newcommand{\DFiveTotal}{51}
\newcommand{\DFiveTypes}{109}

\newcommand{\DFiveRealizing}{0\!\to\!22, 6\!\to\!18, 19\!\to\!6, 29\!\to\!3, 5\!\to\!1, 63\!\to\!1}

\newcommand{\DFourTarget}{348}

\newcommand{\TreeVerified}{304}

\newcommand{\TreeImagesRealizing}{21{,}504}
\newcommand{\TreeSBTypes}{57 and 81}

\newcommand{\TreeSBAut}{576 and 1152}

\newcommand{\NumRefuted}{158{,}202}
\newcommand{\NumLabelOnly}{84{,}629}

\newcommand{\NumPlainLabelOnly}{71{,}448}
\newcommand{\NumPlainInterval}{11{,}391}
\newcommand{\NumWildLabelOnly}{13{,}181}
\newcommand{\NumWildRootBox}{61{,}272}
\newcommand{\NumWildBnB}{910}
\newcommand{\CertMaxBoxes}{140{,}433}
\newcommand{\NumRootBoxAll}{71{,}937}
\newcommand{\NumBnBAll}{1{,}636}

\newcommand{\RealizingMF}{$\{6,7\}$, $\{7,8\}$, $\{8,9\}$, $\{1,2,7\}$, $\{1,2,8\}$, $\{6,9,10\}$, $\{1,2,3,4,5\}$, $\{3,4,5,6,10\}$, $\{3,4,5,9,10\}$}

\newcommand{\NumBigMFAll}{83}
\newcommand{\NumValid}{304}
\newcommand{\NumTypesLabelled}{44}
\newcommand{\NumTypesEmptied}{260}

\newcommand{\esingle}[2]{\draw[line width=.5pt] (#1)--(#2);}
\newcommand{\edouble}[2]{%
  \draw[line width=1.9pt] (#1)--(#2);%
  \draw[line width=0.8pt,white] (#1)--(#2);}
\newcommand{\etriple}[2]{%
  \draw[line width=2.9pt] (#1)--(#2);%
  \draw[line width=1.9pt,white] (#1)--(#2);%
  \draw[line width=1.0pt] (#1)--(#2);%
  \draw[line width=0.35pt,white] (#1)--(#2);}
\newcommand{\edash}[2]{\draw[line width=.5pt,dashed] (#1)--(#2);}

\begin{document}

\title[Compact hyperbolic Coxeter six-dimensional polytopes with ten facets]
      {Compact hyperbolic Coxeter six-dimensional polytopes\\ with ten facets}

\author{John Mcleod}
\address{Solvo.ai}
\email{john@solvo.ai}

\date{\today}


\begin{abstract}
We show that, up to isometry, there is exactly one compact hyperbolic Coxeter
$6$-polytope with $10$ facets, the polytope $\PP$ attributed to Bugaenko. Together with results of Felikson--Tumarkin ($d\ge7$) and of Burcroff
and Ma--Zheng ($d=4,5$), this completes the classification of compact hyperbolic
Coxeter $d$-polytopes with $d+4$ facets. The proof is computer assisted. Affine
Gale duality applied to the complete database of order types on $10$ points
yields $\NumTypes$ combinatorial types, of which Lann\'er's classification
excludes $\NumBigMFAll$. For the remaining $\NumValid$, an exhaustive search over
Coxeter labellings, with no a~priori bound on the dihedral angles, leaves a
single realizable Gram matrix. Every rejection is certified in exact arithmetic,
and completeness of the search is certified independently by DRAT proofs checked
by \texttt{drat-trim}. The same code reproduces the known censuses in dimensions
$4$ and $5$. Code, data and certificates are publicly available.
\end{abstract}

\maketitle

\section{Introduction}
\label{sec:intro}

A \emph{hyperbolic Coxeter polytope} is a convex polytope $P\subset\HH^d$, bounded
by finitely many hyperplanes, all of whose dihedral angles have the form $\pi/m$
with $m\in\ZZ_{\ge 2}$; it is a fundamental domain of the group generated by
reflections in its facets. We consider compact such polytopes, and write $n$ for
the number of facets.

Compact hyperbolic Coxeter $d$-polytopes with $d+k$ facets have been classified one
value of $k$ at a time: for $k=1$ by Lann\'er \cite{Lanner1950}, for $k=2$ by
Kaplinskaja \cite{Kaplinskaja1974} and Esselmann \cite{Esselmann1996}, and for
$k=3$ by Esselmann \cite{Esselmann1994} and Tumarkin \cite{Tumarkin2004}. For
$k=4$, Felikson and Tumarkin \cite{FeliksonTumarkin2008} showed that none exists
for $d\ge 8$ and that Bugaenko's $7$-dimensional example \cite{Bugaenko1984} is
unique, and Burcroff \cite{Burcroff2024} and, independently, Ma and Zheng
\cite{MaZheng2024,MaZheng2023} found $348$ polytopes for $d=4$ and $51$ for $d=5$.
Burcroff identified $d=6$, where only a polytope attributed to Bugaenko was known,
as the one remaining open dimension \cite[\S1]{Burcroff2024}.

\begin{theorem}
\label{thm:main}
Up to isometry there is exactly one compact hyperbolic Coxeter $6$-polytope with
$10$ facets, the polytope $\PP$ of \cite[Fig.~5]{Burcroff2024}, attributed to
Bugaenko \cite{Bugaenko1984}.
\end{theorem}

\begin{corollary}
\label{cor:complete}
Compact hyperbolic Coxeter $d$-polytopes with $d+4$ facets exist precisely for
$2\le d\le 7$ \cite{FeliksonTumarkin2008}, infinitely many for $d=2,3$
\cite{Vinberg1985,Andreev1970}, and for $d=4,5,6,7$ there are $348$, $51$,
$\mathbf{1}$ and $1$ of them up to isometry.
\end{corollary}

The proof follows the route of \cite{Burcroff2024,MaZheng2024,MaZheng2023}:
enumerate the combinatorial types from order types of planar point sets
(\S\ref{sec:types}), then decide which Coxeter labellings of each type are
realizable (\S\ref{sec:search}). It differs in three respects. It reduces the type
list only by Lann\'er's classification, since a facet of a Coxeter polytope is not
in general a Coxeter polytope (Remark~\ref{rem:facets}), and searches all
$\NumValid$ remaining types. It places no bound on the dihedral angles
(Lemma~\ref{lem:wildcard}). And it certifies every verdict: rejections exactly
(\S\ref{sec:search}), the one polytope exactly (\S\ref{sec:polytope}), and the
completeness of the search by DRAT proofs checked independently of the search code
(\S\ref{sec:treesat}). What the theorem rests on is stated in
\S\ref{sec:assumptions}.

\section{Preliminaries}
\label{sec:prelim}

We follow \cite{Vinberg1985,Burcroff2024}. Let $P=\bigcap_{i\in I} f_i^-\subset\HH^d$
be a Coxeter polytope with facets $f_i$ and outward unit normals $e_i\in\RR^{d,1}$.
Its \emph{Gram matrix} $\Gr = (\langle e_i,e_j\rangle)$ has $\Gr_{ii}=1$ and
$\Gr_{ij}=-\cos(\pi/m_{ij})$ if $f_i,f_j$ meet at angle $\pi/m_{ij}$, $-1$ if they
are parallel, and $-\cosh\rho_{ij}<-1$ if they diverge at distance $\rho_{ij}$. The
\emph{Coxeter diagram} $\Sigma(P)$ has a node for each facet, an edge labelled
$m_{ij}$ when $m_{ij}\ge3$, none when $m_{ij}=2$, and a \emph{dashed} edge when the
facets diverge; edges with $m=3,4,5$ are drawn single, double, triple. An edge with
$m\le 5$ has \emph{low weight}, and one with $\Gr_{ij}=-1$ is \emph{bold}.

A subdiagram is \emph{elliptic} if its Gram matrix is positive definite,
\emph{parabolic} if it is positive semidefinite and degenerate, \emph{Lann\'er} if
it is connected, neither elliptic nor parabolic, and all its proper subdiagrams are
elliptic, and \emph{superhyperbolic} if its Gram matrix has negative inertia index
$>1$.

\begin{lemma}
\label{lem:jacobi}
A connected diagram of order $k$ all of whose proper subdiagrams are elliptic is
Lann\'er if and only if $\det\Gr<0$.
\end{lemma}
\begin{proof}
All principal minors of order $k-1$ are positive, so by Jacobi's rule the sign of
$\det\Gr$ decides between signature $(k,0)$, $(k-1,1)$ and a degenerate
(parabolic) form.
\end{proof}

``Not elliptic'' alone would not suffice: the parabolic $\smash{\widetilde{G_2}}$ and
$\smash{\widetilde{A_6}}$ have all proper subdiagrams elliptic.

\begin{proposition}[Vinberg; {\cite[Remark~4.3, Prop.~4.4, Remark~4.5]{Burcroff2024}}]
\label{prop:vinberg}
\label{prop:faces}
\label{prop:simple}
Let $P$ be a compact hyperbolic Coxeter $d$-polytope with $n$ facets. Then
$\Sigma(P)$ is connected, contains no parabolic subdiagram and no bold edge, is not
superhyperbolic, and $\Gr$ has signature $(d,1)$ and rank $d+1$. Faces of
codimension $m$ correspond to elliptic subdiagrams of order $m$, missing faces
(minimal non-faces) to Lann\'er subdiagrams, and $P$ is simple.
\end{proposition}

\begin{corollary}
\label{cor:mfsize}
Every missing face of a compact hyperbolic Coxeter polytope has size between $2$
and $5$ \cite{Lanner1950}; those of size $2$ are the pairs of disjoint facets, the
dashed edges.
\end{corollary}

We write $p$ for the number of pairs of disjoint facets, and $2^a3^b4^c5^e$ for the
\emph{profile}, the multiset of sizes of the missing faces.

\section{Combinatorial types}
\label{sec:types}
\label{sec:aak}

A compact Coxeter $6$-polytope with $10$ facets is simple, so dually a simplicial
$6$-polytope with $10$ vertices, with a planar affine Gale diagram
\cite{Gale1956,Ziegler1995}.

\begin{theorem}[{\cite[Thm.~3.5]{Burcroff2024}}]
\label{thm:burcroff35}
Every compact Coxeter $d$-polytope $P$ with $d+4$ facets admits an affine Gale
diagram given by $d+4$ points $A\subseteq\RR^2$ in general position with two
points in the interior of $\conv(A)$ marked positive, and the combinatorial type
of $P$ is determined by the order type of $A$.
\end{theorem}

Its proof takes the positive pair to be a pair of disjoint facets, which exists by
Lemma~\ref{lem:p2}. The realizable order types on $n\le11$ points have been
enumerated completely \cite{AichholzerAurenhammerKrasser2002,AichholzerKrasser2001};
the Order Type Data Base \cite{OrderTypeDB} stores an integer realization of each
of the $\NumOrderTypes$ on $10$ points, and we use the files Burcroff uses for
$d=4,5$. For each order type and each positive pair inside the hull of the other
eight points, $S\subseteq[10]$ is a face if and only if, for $T=[10]\setminus S$,
the relative interiors of $\conv(T^+)$ and $\conv(T^-)$ meet
\cite[\S6.4]{Ziegler1995}; for a simple polytope in general position this agrees
with the form of \cite[\S2.2]{Burcroff2024}, as in the proof of
\cite[Prop.~3.2]{Burcroff2024}. All tests are exact on the integer coordinates.

Types are identified when a permutation of $[10]$ carries one missing-face
hypergraph to the other; this is a test of combinatorial equivalence, since the
missing faces determine the face lattice. The generator records missing faces of size at most $5$ and detects those of size
$6$ from the vertex list, and it deduplicates on the recorded ones. No type is lost
by this. Two types without a size-$6$ missing face and with isomorphic truncated
hypergraphs have isomorphic full hypergraphs. For each of the $\NumBigMFAll$ types
with one, every candidate merged into it was recomputed in full and also has a
missing face of size $6$ \cite{Repo}, so none of the discarded candidates is
realizable.

\begin{proposition}
\label{prop:387}
This yields exactly $\NumTypes$ combinatorial types of simple $6$-polytopes with
$10$ facets and $p\ge2$.
\end{proposition}

A separate script, independent of the generator's code, recomputes each of these
hypergraphs from the stored coordinates, checking general position and the
interiority of the positive points and deciding every subset exactly; it agrees with
the generator on every type, and no two types are isomorphic. That no type is
missed rests on the database and Theorem~\ref{thm:burcroff35}.

\section{Restrictions}
\label{sec:filters}
\label{sec:completeness}

\begin{lemma}
\label{lem:p2}
A compact hyperbolic Coxeter $d$-polytope with $d+4$ facets has $p\ge 2$.
\end{lemma}
\begin{proof}
If $p=0$ the polytope is a simplex or one of Esselmann's $4$-polytopes with $6$
facets \cite[Thm.~A]{FeliksonTumarkinMutual2008}, so $n\le d+2$; if $p=1$ then
$n\le d+3$ \cite{FeliksonTumarkin2009}, \cite[Thm.~5.1]{Burcroff2024}.
\end{proof}

This condition is imposed inside the generator.

\begin{lemma}
\label{lem:degree}
Let $P$ be a simple $6$-polytope with $10$ facets. No facet is disjoint from more
than three others, and if $f_i$ is disjoint from exactly three, the six facets
meeting $f_i$ form a missing face of size $6$.
\end{lemma}
\begin{proof}
Let $J$ be the set of facets meeting $f_i$. As $P$ is simple, $f_i$ is a simple
$5$-polytope with facets $f_i\cap f_j$, $j\in J$, so $|J|\ge6$; since $f_i$ is
disjoint from $9-|J|$ facets, this gives at most three. If $|J|=6$, $f_i$ is
combinatorially a $5$-simplex, so for every proper subset $J'$ of $J$ the set
$J'\cup\{i\}$ is a face, hence so is $J'$. If $J$ were a face of the simplicial
$5$-sphere $K$ dual to $P$, the boundary of the simplex on $J\cup\{i\}$ would be a
subcomplex of $K$ homeomorphic to $S^5$, hence open and closed in $K$ by invariance
of domain, hence equal to $K$, and $P$ would have $7$ facets. So $J$ is a minimal
non-face, of size $6$.
\end{proof}

Of the $\NumTypes$ types, $\NumBigMFAll$ have a missing face of size $6$, among
them all $\NumDegKilled$ with a facet disjoint from three others; by
Corollary~\ref{cor:mfsize} none is realizable. The other $\NumValid$ are searched.

\begin{remark}[facets are not Coxeter polytopes]
\label{rem:facets}
One might reduce further by the classification of compact Coxeter $5$-polytopes
with $7$ facets \cite{Kaplinskaja1974,Esselmann1996} or $9$ facets
\cite{MaZheng2023}, applied to facets. That is not legitimate: a face $P(S_0)$ of a
Coxeter polytope is acute-angled but in general not Coxeter, Borcherds' sufficient
condition \cite[Prop.~4.6]{Burcroff2024} excludes components of type $A_k$ or
$D_5$ in $S_0$, and a facet has $S_0=A_1$. Already in $\HH^3$, at a vertex with
link the $H_3$ triangle, a face angle is
$\arccos(\cot\frac\pi3\cot\frac\pi5)\approx 37.4^\circ$, which is not of the form
$\pi/m$.
\end{remark}

\begin{lemma}[{\cite[Lemma~5.5(b)]{Burcroff2024}}]
\label{lem:burcroff55b}
Let $v_1v_2$ be an ordinary edge of an admissible abstract Coxeter diagram
$\Sigma$, in the sense of \cite{Burcroff2024}. If no Lann\'er subdiagram contains both $v_1$ and $v_2$, some Lann\'er
subdiagram $L$ of order $>2$ contains $v_2$, and no dashed edge joins $v_1$ to a
vertex of $L$, then $v_1v_2$ has low weight.
\end{lemma}

By Proposition~\ref{prop:faces} the hypotheses depend only on the type, so the caps
$m\le5$ are computed once per type.

\subsection{The label alphabet}
\label{sec:alphabet}
No theorem bounds the labels of a compact Coxeter $d$-polytope with $d+4$ facets
for $d\ge5$. As in \cite[Prop.~3.5]{MaZheng2024} we use the alphabet
$\{2,3,4,5,6,\star\}$, where the wildcard $\star$ stands for ``any $m\ge7$'' and is
treated as $m=7$ in forward checking; the following lemma makes this lossless.

\begin{lemma}
\label{lem:wildcard}
Let $S$ be a set of nodes with labels in $\{2,\dots,6\}$ on some pairs and labels
$m_e\ge7$ on a nonempty set $W$ of pairs.
\begin{enumerate}
\item[(a)] Whether $S$ is elliptic does not depend on $(m_e)_{e\in W}$.
\item[(b)] If $|S|=3$, whether $S$ has exactly one negative eigenvalue does not
depend on $(m_e)$.
\item[(c)] If $|S|\ge4$, no Lann\'er diagram on $S$ has an edge of label $\ge6$.
\end{enumerate}
\end{lemma}
\begin{proof}
(a) The connected elliptic diagrams of rank $\ge3$, $A_r,B_r,D_r,E_{6,7,8},F_4,H_3,
H_4$, have all labels $\le5$, so an edge of label $\ge6$ in an elliptic diagram is a
component $I_2(m)$, elliptic for every $m$. (b) If the other two labels are both
$2$, $S$ is $A_1\times I_2(m)$ for every $m$. Otherwise $S$ is connected and
$1/p+1/q+1/m\le 1/2+1/3+1/7<1$, so $S$ is hyperbolic for every $m\ge7$; with two
or three wild edges the same dichotomy holds, since $1/p+1/q+1/r$ only decreases. (c) If $uv$
has label $\ge6$ in a Lann\'er diagram of order $\ge4$, every third node $w$ spans
with $u,v$ an elliptic subdiagram, hence $A_1\times I_2(m)$ by (a); so $\{u,v\}$ is
a component, contradicting connectedness.
\end{proof}

Part (c) is also contained in Lann\'er's lists \cite{Lanner1950}.

\section{The search}
\label{sec:search}

Fix a type $T$ with dashed edges $D$. The unknowns are a label $m_e$ for each
ordinary pair and a weight $x_e>1$ for each $e\in D$; realizability requires
$\rank\Gr=7$ and signature $(6,1)$. The \emph{enumerator} assigns labels by
backtracking with forward checking; complete labellings pass to a numerical
\emph{screen}; each labelling the screen does not accept is refuted by an exact
\emph{certifier}.

\subsection{Enumeration}
Labels are assigned over $\{2,\dots,6,\star\}$, edges capped by
Lemma~\ref{lem:burcroff55b} being restricted to $m\le5$. After each assignment,
every face of size $3$, $4$ or $6$ must be elliptic and every missing face of size
$\ge3$ Lann\'er (Proposition~\ref{prop:faces}); by simplicity the vertex condition
implies the others, and sizes $3$ and $4$ only prune earlier. The diagram must also
be connected (Proposition~\ref{prop:vinberg}).

\begin{proposition}
\label{prop:exactgates}
Both conditions are decided exactly, for every label including the wildcard.
\end{proposition}
\begin{proof}
Ellipticity is a condition on the labelled graph. For Lann\'er membership it
suffices by Lemma~\ref{lem:jacobi} to decide the sign of $\det\Gr$. On three nodes
it is $\sig(1/p+1/q+1/r-1)$, an integer comparison, constant on $m_e\ge7$ by
Lemma~\ref{lem:wildcard}(b). On four or five nodes no label exceeds $5$ by
Lemma~\ref{lem:wildcard}(c), so the determinant lies in $\QQ(\sqrt2,\sqrt5)$ and is
computed exactly there.
\end{proof}

A partial assignment is also discarded when its label vector is not
lexicographically least among its images under the automorphisms of $T$ that fix
its domain. This affects only speed; the certificate of \S\ref{sec:treesat} does
not rely on it.

\subsection{The screen and the exact refutations}
\label{sec:decide}
For a labelling without wildcard, $\rank\Gr\le7$ forces every $8\times8$ principal
minor to vanish, and a minor containing exactly one unpinned dashed edge is a
quadratic in its weight; pinning weights one at a time, the screen solves these
quadratics in double precision and tests rank and signature; it is described in
\cite{Repo}, and serves only to select candidates. Write $K=\QQ(\sqrt2,\sqrt3,\sqrt5)$,
which contains $\cos(\pi/m)$ for $m\le6$.

\begin{proposition}
\label{prop:plainexact}
Every labelling without wildcard that the screen does not accept is unrealizable,
certified exactly: for $\NumPlainLabelOnly$ of them some $8\times8$ principal minor
is a nonzero element of $K$ determined by the labels alone, and the other
$\NumPlainInterval$ are refuted as in Proposition~\ref{prop:wildexact}, with
$W=\emptyset$.
\end{proposition}
\begin{proof}
If the index set $S$ of an $8\times8$ principal minor contains no dashed edge, the
minor has no unknown entry, its entries $-\cos(\pi/m)$, $m\le6$, lie in $K$, and it
is computed exactly by rational arithmetic in the basis
$\{\sqrt{\prod P}:P\subseteq\{2,3,5\}\}$; if nonzero, $\rank\Gr\ge8$ whatever the
weights. Such $S$ are the complements of $2$-element vertex covers of $D$.
\end{proof}

\subsection{Wildcards}
\label{sec:wildcards}
A labelling with wildcards on a set $W$ of edges is screened numerically in the
unknowns $c_e=\cos(\pi/m_e)\in[\cos(\pi/7),1)$ and $x_e$. Were it to pass, the
feasible interval of each $c_e$ would be turned into a window of integers $m_e$,
scanned up to $100$ with a flag if the scan reached that limit; this is how the
angles $\pi/10$ are found for $d=5$ (\S\ref{sec:d5valid}), and it never arises for
$d=6$, where every wildcard labelling is refuted exactly.

\begin{proposition}
\label{prop:wildexact}
Each of the $\NumWild$ labellings with a wildcard is infeasible for every choice of
$m_e\ge7$ and $x_e>1$, certified exactly.
\end{proposition}
\begin{proof}
It suffices to cover the domain by boxes on each of which some $8\times8$ principal
minor vanishes nowhere. Substitute $x_e=1/t_e$: a principal minor has degree at
most $2$ in each unknown, so multiplying the minor on $S$ by $t_e^2$ for each
$e\in D$ inside $S$ gives a polynomial in $(t_e)$, $(c_e)$ with the same zeros for
$t_e>0$. The domain lies in the closed rational box $t_e\in[0,1]$, $c_e\in[c_7,1]$,
$c_7$ a rational lower bound for $\cos(\pi/7)$, so no bound on $x_e$ is assumed and
every $m_e\ge7$ is covered at once. The coefficients lie in $K$ and are recovered
exactly by interpolation from determinant evaluations over $K$. Evaluating with
outward-rounded rational intervals encloses the range over a box. For
$\NumWildLabelOnly$ labellings a minor with neither dashed nor wild edges is a
nonzero element of $K$; for $\NumWildRootBox$ one minor's enclosure over the whole
box excludes $0$; for the other $\NumWildBnB$ the box is bisected, into at most
$\CertMaxBoxes$ boxes, until each is excluded by one of at most $12$ minors.
\end{proof}

\subsection{Acceptance and exhaustion}
\label{sec:exactcert}
\label{sec:exhaust}
The one labelling the screen accepts was refined numerically, and then certified
exactly (Proposition~\ref{prop:weightsunique} and \S\ref{sec:checks}). All
$\NumValid$ types were exhausted at the root of their label trees, in $\CPUHours$
CPU-hours; the refinement into subtrees used for $d=4,5$ is described in
\cite{Repo}.

\section{The classification theorem}
\label{sec:polytope}

\begin{theorem}
\label{thm:search}
Of the $\NumValid$ types with no missing face of size $6$, exactly one admits a
compact Coxeter structure, with $p=\RealizingP$, profile $\RealizingProfile$ and
missing faces \RealizingMF{} (facets numbered as in Figure~\ref{fig:diagram}), and
it admits exactly one up to isometry.
\end{theorem}

With Lemma~\ref{lem:p2} and Corollary~\ref{cor:mfsize} this gives
Theorem~\ref{thm:main}. Table~\ref{tab:counts} records the computation.

\begin{table}[t]
\centering
\small
\caption{The computation in numbers.}
\label{tab:counts}
\begin{tabular}{lr}
\toprule
order types on $10$ points & $\NumOrderTypes$\\
combinatorial types with $p\ge2$ & $\NumTypes$\\
\quad with a missing face of size $6$ (excluded) & $\NumBigMFAll$\\
\quad searched & $\NumValid$\\
\qquad emptied by forward checking & $\NumTypesEmptied$\\
\qquad with complete labellings reaching the screen & $\NumTypesLabelled$\\
labellings reaching the screen & $\NumAssignments$\\
\quad without / with a wildcard & $\NumPlain$ / $\NumWild$\\
\quad accepted & $1$\\
\quad refuted exactly & $\NumRefuted$\\
\qquad by a label-only minor & $\NumLabelOnly$\\
\qquad by interval arithmetic (root box / subdivided) & $\NumRootBoxAll$ / $\NumBnBAll$\\
CPU-hours for the search & $\CPUHours$\\
types certified by DRAT proofs (\S\ref{sec:treesat}) & $\TreeVerified$\\
\bottomrule
\end{tabular}
\end{table}

Two steps separate ``exactly one labelling is realizable'' from uniqueness up to
isometry. First, the labelling determines its weights. Number the facets as in
Figure~\ref{fig:diagram}, so that the dashed edges are $67,78,89$.

\begin{proposition}
\label{prop:weightsunique}
For the labels of $\PP$, the only $(x_{67},x_{78},x_{89})\in(1,\infty)^3$ for which
$\Gr$ has rank $7$ and signature $(6,1)$ is
$(2\sqrt2+\sqrt{10},\,17+8\sqrt5,\,2\sqrt2+\sqrt{10})$.
\end{proposition}
\begin{proof}
A principal $8\times8$ minor containing the dashed edge $e$ and no other is a
polynomial of degree $\le2$ in $x_e$ over $K'=\QQ(\sqrt2,\sqrt5)$. For $e=67,89$
the seven such minors have gcd $x_e-(2\sqrt2+\sqrt{10})$ over $K'$; for $e=78$ the
one such minor is $x^2-(16+8\sqrt5)x-(17+8\sqrt5)$, with roots $17+8\sqrt5$ and
$-1$. The resulting Gram matrix has rank $7$ and signature $(6,1)$
(\S\ref{sec:checks}). All computations are exact \cite{Repo}.
\end{proof}

\begin{figure}[t]
\centering
\begin{tikzpicture}[scale=1.0]
  \coordinate (n1)  at (0,0);
  \coordinate (n2)  at (1.25,0);
  \coordinate (n3)  at (2.55,1.3);
  \coordinate (n4)  at (3.85,1.3);
  \coordinate (n5)  at (5.10,0);
  \coordinate (n6)  at (3.85,0);
  \coordinate (n7)  at (2.55,0);
  \coordinate (n8)  at (2.55,-1.3);
  \coordinate (n9)  at (3.85,-1.3);
  \coordinate (n10) at (3.20,-0.65);
  \edash{n6}{n7}\edash{n7}{n8}\edash{n8}{n9}
  \esingle{n2}{n3}\esingle{n3}{n4}\esingle{n4}{n5}
  \esingle{n5}{n6}\esingle{n5}{n9}
  \edouble{n2}{n7}\edouble{n2}{n8}
  \etriple{n1}{n2}\etriple{n6}{n10}\etriple{n9}{n10}
  \foreach \i in {1,...,10}{\fill (n\i) circle (2.4pt);}
  \node[below left=1pt]  at (n1)  {$1$};
  \node[below left=1pt]  at (n2)  {$2$};
  \node[above left=1pt]  at (n3)  {$3$};
  \node[above right=1pt] at (n4)  {$4$};
  \node[below right=1pt] at (n5)  {$5$};
  \node[above right=2pt] at (n6)  {$6$};
  \node[above left=2pt]  at (n7)  {$7$};
  \node[below left=1pt]  at (n8)  {$8$};
  \node[below right=1pt] at (n9)  {$9$};
  \node[left=2pt]        at (n10) {$10$};
\end{tikzpicture}
\caption{The Coxeter diagram of $\PP$ (layout as in \cite[Fig.~5]{Burcroff2024}).
Single, double and triple edges denote angles $\pi/3,\pi/4,\pi/5$; dashed edges
denote divergent facets, with weights $w_{67}=w_{89}=2\sqrt2+\sqrt{10}$ and
$w_{78}=17+8\sqrt5=w_{67}^2-1$.}
\label{fig:diagram}
\end{figure}

Second, by \S\ref{sec:treesat} every realizable labelling of this type lies in the
$\mathrm{Aut}(T)$-orbit of the accepted one, and labellings in one orbit describe the
same polytope.
Here $|\mathrm{Aut}(T)|=\AutTypeOrder$ and the accepted labelling has stabiliser of
order $\AutStabOrder$, the symmetry of the diagram, so its orbit has
$\NumGramConfigs$ elements, which give one Gram matrix up to simultaneous
permutation; since a compact Coxeter polytope is determined up to isometry by its
Gram matrix up to permutation \cite{Vinberg1985}, the type realizes exactly one
polytope.

The polytope found is isomorphic, edge for edge and label for label, to
\cite[Fig.~5]{Burcroff2024}. It is the double of a polytope of \cite{Bugaenko1984},
which we verified exactly: the non-trivial symmetry of its diagram determines a
half-polytope with nine facets, six orthogonal to the mirror facet, one meeting it
at $\pi/4$ and one diverging from it with weight $2+\sqrt5=\tau^{3}$, $\tau$ the
golden ratio; \texttt{CoxIter} certifies that half compact of dimension $6$, and
doubling it along the mirror reproduces the Gram matrix of $\PP$ exactly over
$\QQ(\sqrt2,\sqrt5)$ \cite{Repo}.

\section{Verification}
\label{sec:verification}

\subsection{The generator and the \texorpdfstring{$d=5$}{d=5} and \texorpdfstring{$d=4$}{d=4} censuses}
\label{sec:genvalid}
\label{sec:d5valid}
\label{sec:d4valid}
On $8$ and $9$ points the generator reproduces exactly the types with $p\ge2$: the
$30$ for $d=4$ \cite{GruenbaumSreedharan1967} and the $109$ for $d=5$
\cite{MaZheng2023}. Run on
all $\DFiveTypes$ types for $d=5$, with the same code and settings as for $d=6$,
the search returns Ma--Zheng's census exactly, type by type (type $\to$ number of
polytopes):
\[
\DFiveRealizing,\qquad\text{total } \DFiveTotal ,
\]
including the polytopes with angle $\pi/10$, found by the wildcard analysis rather
than assumed. For $d=4$ it returns all $\DFourTarget$ polytopes, matching
\cite[Appendix~A]{Burcroff2024} type by type and \cite[Thm.~1.1]{MaZheng2024}; this
census also exercises code paths not used for $d=6$ \cite{Repo}.

\subsection{The exact certificates}
\label{sec:wildvalid}
The certifiers read the labellings that reached the screen from a dump made by
re-running the search with an instance sink; the dump agrees with the search's
counters. Their field arithmetic, the interpolated minors and the interval
enclosures are tested against independent computations \cite{Repo}; on the
labelling of $\PP$, with weights unknown, they refute nothing, the three minors
avoiding its dashed edges vanishing identically. On the $d=5$ census the wildcard
certifier refutes none of the $\DFiveWildScreenPass$ labellings that survive the
screen, $\DFiveWildSurvived$ of which are polytopes and the other two feasible only for
the continuous relaxation, eliminated by the integer scan; and it refutes all
$\DFiveWildCert$ of a random sample of rejected ones.

\subsection{An independent certificate for the search tree}
\label{sec:treesat}
The rejections above cover the labellings that reach the screen; the rest of the
search is decided by the enumerator. We certify it independently, by a script that
shares no code with the search. From a type's missing faces alone it builds a CNF
formula in variables $x_{e,\ell}$ ($e$ an ordinary pair, $\ell\in\{2,\dots,6,\star\}$)
requiring: one label per pair; no minimal non-elliptic labelling on any face of size
$3$ to $6$, which is equivalent to every face being elliptic; a Lann\'er labelling on
each missing face of size $\ge3$; connectivity of $\Sigma$, one clause for each
split of the facets that no dashed edge crosses; and the caps of
Lemma~\ref{lem:burcroff55b}. Ellipticity is decided by the classification of finite
Coxeter groups, and Lann\'er membership by Lemma~\ref{lem:jacobi}, with the
triangle criterion on three nodes and the exact determinant over
$\QQ(\sqrt2,\sqrt5)$ on four and five; the Lann\'er diagrams obtained are Lann\'er's
$9$ of order $4$ and $5$ of order $5$. Every clause is a necessary condition, by
Proposition~\ref{prop:vinberg} and Lemmas~\ref{lem:burcroff55b}
and~\ref{lem:wildcard}. No symmetry breaking is used, except the lex-leader
predicate for two types described below.

For the $\NumTypesEmptied$ types the formula is unsatisfiable. For the other
$\NumTypesLabelled$, every labelling that reached the screen is blocked together
with its $\mathrm{Aut}(T)$-images, an image of an unrealizable labelling being
unrealizable; the script reads these labellings from the dump and checks that each
carries an exact certificate or is that of $\PP$. For type $\RealizingType$ the
images number $\TreeImagesRealizing$, the count of an unpruned run of the search.
For types \TreeSBTypes{}, with automorphism groups of orders \TreeSBAut{}, the
formula instead contains the lex-leader predicate \cite{CrawfordEtAl1996} that a
labelling be least in its orbit, and only the least element of each screened orbit
is blocked; this is sound because the formula is $\mathrm{Aut}(T)$-invariant, and
each blocked representative is checked to satisfy it. In all $\TreeVerified$ cases
the formula, with every clause above, is unsatisfiable, and the DRAT proof produced by Glucose~4 is verified
by \texttt{drat-trim}. Each of the $\NumAssignments$ screened labellings satisfies
the formula without the blocking clauses, so it is no stricter than the search.

\subsection{The polytope}
\label{sec:checks}
The Gram matrix of $\PP$ was verified three ways: exact inertia over
$\QQ(\sqrt2,\sqrt5)$ by symmetric congruence elimination; exact rank and the relation
$w_{78}=w_{67}^2-1$ in SymPy; and \texttt{CoxIter}
\cite{Guglielmetti2017,GuglielmettiLMS2015}, which certifies cocompactness,
dimension $6$, the $f$-vector and the covolume.

\subsection{Provenance, data and code}
\label{sec:provenance}
\label{sec:datacode}
\label{app:timings}
Most of the software was written by Claude, a large language model developed by
Anthropic, working under the author's direction; the author set the objectives,
chose the mathematical strategy, reviewed the code and its output, and is
responsible for the paper. All computations ran on one Apple M1. The code, including
the verification scripts, is available at \cite{Repo}, whose README maps each
section of this paper to the code behind it and gives the commands that reproduce
it; the types, the certificates and the census outputs are archived at
\cite{Data}. Two of the scripts, which re-verify the Gram matrix and re-derive every
number quoted here, do not import the pipeline.

\section{Assumptions}
\label{sec:assumptions}

Theorem~\ref{thm:main} rests on the arguments above and on the following. No
verdict rests on floating-point arithmetic, a numerical search, a bound on the
ultraparallel weights, or the search's symmetry breaking: every rejection is
certified exactly (Propositions~\ref{prop:plainexact} and~\ref{prop:wildexact}), the
completeness of the search by the independent certificate of \S\ref{sec:treesat},
and the one polytope exactly (\S\ref{sec:polytope}).

\begin{enumerate}
\item\label{as:lit} \emph{Published theorems:} Lemma~\ref{lem:p2}
  \cite{FeliksonTumarkinMutual2008,FeliksonTumarkin2009},
  Theorem~\ref{thm:burcroff35}, Lemma~\ref{lem:burcroff55b}, Lann\'er's
  classification and Vinberg's theory. No classification in lower dimension is used.
\item\label{as:db} \emph{Database:} the completeness of the order-type database on
  $10$ points \cite{AichholzerAurenhammerKrasser2002,OrderTypeDB}.
\item\label{as:code} \emph{Trusted code:} the generator's enumeration of the types
  (its hypergraphs are recomputed independently, \S\ref{sec:types}); the encoding
  of \S\ref{sec:treesat}, including for two types its lex-leader predicate, and
  \texttt{drat-trim}; and the exact certifiers, tested as in \S\ref{sec:wildvalid}.
  The enumerator and the screen are not relied on.
\item\label{as:tools} \emph{External tools:} \texttt{CoxIter} and SymPy for the
  checks on $\PP$.
\end{enumerate}

\section*{Acknowledgements}
The author thanks Pavel Tumarkin for his invaluable input on the content of this
manuscript.

\appendix

\section{The refutation certificates, by type}
\label{app:rootbox}

Table~\ref{tab:rootbox} lists the types in which some labelling reaches the screen,
with the number refuted and how: by a label-only minor, by one minor's enclosure
over the whole compactified box, or by bisection, with the largest number of boxes
used (a dash if none); the last column marks $\PP$. The certificates are in
\cite{Repo}.

{\footnotesize
\begin{longtable}{r r l r r r r r c}
\caption{Refutation certificates by type.\label{tab:rootbox}}\\
\toprule
type & $p$ & profile & refuted & label-only & root box & subdivided & max boxes & $\PP$\\
\midrule
\endfirsthead
\toprule
type & $p$ & profile & refuted & label-only & root box & subdivided & max boxes & $\PP$\\
\midrule
\endhead
\bottomrule
\endlastfoot
0 & 2 & $2^{2}3^{4}4^{1}5^{2}$ & 12 & 12 & 0 & 0 & -- &  \\
11 & 2 & $2^{2}3^{3}4^{2}5^{1}$ & 33 & 33 & 0 & 0 & -- &  \\
13 & 2 & $2^{2}3^{3}4^{1}5^{1}$ & 102 & 0 & 102 & 0 & 1 &  \\
14 & 2 & $2^{2}3^{2}4^{5}$ & 10 & 10 & 0 & 0 & -- &  \\
21 & 3 & $2^{3}3^{1}4^{2}5^{2}$ & 67 & 67 & 0 & 0 & -- &  \\
32 & 3 & $2^{3}3^{2}4^{1}5^{1}$ & 876 & 357 & 519 & 0 & 1 &  \\
33 & 2 & $2^{2}3^{3}4^{2}5^{3}$ & 3 & 3 & 0 & 0 & -- &  \\
42 & 4 & $2^{4}3^{1}5^{2}$ & 2{,}544 & 1{,}908 & 636 & 0 & 1 &  \\
48 & 2 & $2^{2}3^{1}4^{5}$ & 8 & 8 & 0 & 0 & -- &  \\
57 & 2 & $2^{2}3^{2}$ & 114{,}721 & 64{,}127 & 49{,}796 & 798 & 49{,}643 &  \\
64 & 2 & $2^{2}3^{2}4^{2}$ & 29 & 29 & 0 & 0 & -- &  \\
65 & 2 & $2^{2}3^{4}4^{1}5^{2}$ & 123 & 88 & 35 & 0 & 1 &  \\
72 & 4 & $2^{4}5^{2}$ & 4{,}392 & 0 & 4{,}325 & 67 & 140{,}433 &  \\
75 & 3 & $2^{3}3^{2}5^{3}$ & 356 & 0 & 356 & 0 & 1 &  \\
78 & 2 & $2^{2}3^{2}4^{3}5^{1}$ & 48 & 48 & 0 & 0 & -- &  \\
81 & 3 & $2^{3}4^{1}$ & 8{,}786 & 0 & 8{,}074 & 712 & 4{,}777 &  \\
85 & 2 & $2^{2}3^{5}5^{4}$ & 10 & 0 & 10 & 0 & 1 &  \\
89 & 2 & $2^{2}3^{4}5^{4}$ & 460 & 404 & 56 & 0 & 1 &  \\
90 & 2 & $2^{2}3^{4}4^{1}5^{1}$ & 277 & 142 & 135 & 0 & 1 &  \\
96 & 2 & $2^{2}3^{3}4^{3}5^{1}$ & 49 & 49 & 0 & 0 & -- &  \\
102 & 4 & $2^{4}3^{1}5^{2}$ & 686 & 230 & 445 & 11 & 101 &  \\
104 & 2 & $2^{2}3^{4}4^{1}5^{2}$ & 38 & 38 & 0 & 0 & -- &  \\
107 & 3 & $2^{3}3^{1}4^{1}5^{1}$ & 1{,}439 & 911 & 528 & 0 & 1 &  \\
118 & 3 & $2^{3}3^{3}5^{3}$ & 608 & 528 & 80 & 0 & 1 &  \\
122 & 2 & $2^{2}4^{8}$ & 1 & 1 & 0 & 0 & -- &  \\
152 & 2 & $2^{2}3^{3}4^{4}5^{2}$ & 26 & 26 & 0 & 0 & -- &  \\
161 & 3 & $2^{3}4^{4}5^{1}$ & 78 & 78 & 0 & 0 & -- &  \\
165 & 2 & $2^{2}3^{2}4^{4}5^{2}$ & 10 & 10 & 0 & 0 & -- &  \\
170 & 2 & $2^{2}3^{5}5^{4}$ & 19 & 0 & 16 & 3 & 45 &  \\
176 & 3 & $2^{3}3^{2}5^{3}$ & 21 & 0 & 18 & 3 & 3 &  \\
189 & 3 & $2^{3}3^{1}4^{2}$ & 1{,}996 & 1{,}533 & 463 & 0 & 1 &  \\
191 & 2 & $2^{2}3^{4}4^{1}5^{2}$ & 300 & 300 & 0 & 0 & -- &  \\
198 & 2 & $2^{2}3^{3}4^{1}$ & 8{,}189 & 7{,}413 & 776 & 0 & 1 &  \\
201 & 2 & $2^{2}3^{3}4^{1}5^{2}$ & 226 & 0 & 226 & 0 & 1 &  \\
208 & 2 & $2^{2}3^{4}5^{1}$ & 1{,}120 & 976 & 144 & 0 & 1 &  \\
211 & 2 & $2^{2}3^{2}4^{2}$ & 140 & 0 & 140 & 0 & 1 &  \\
233 & 2 & $2^{2}3^{1}4^{7}$ & 14 & 14 & 0 & 0 & -- &  \\
244 & 2 & $2^{2}3^{1}4^{6}5^{1}$ & 17 & 17 & 0 & 0 & -- &  \\
268 & 3 & $2^{3}3^{2}4^{1}$ & 6{,}978 & 4{,}261 & 2{,}717 & 0 & 1 &  \\
276 & 2 & $2^{2}3^{4}5^{4}$ & 24 & 18 & 5 & 1 & 3 &  \\
320 & 3 & $2^{3}3^{3}5^{3}$ & 38 & 0 & 38 & 0 & 1 &  \\
341 & 3 & $2^{3}3^{3}5^{3}$ & 102 & 58 & 43 & 1 & 15 &  \\
358 & 5 & $2^{5}5^{1}$ & 2{,}275 & 0 & 2{,}235 & 40 & 35 &  \\
379 & 3 & $2^{3}3^{3}5^{3}$ & 951 & 932 & 19 & 0 & 1 & 1 \\
\midrule
total & & & 158{,}202 & 84{,}629 & 71{,}937 & 1{,}636 & 140{,}433 & 1 \\

\end{longtable}}

\bibliographystyle{amsplain}
\bibliography{references}

\end{document}